\documentclass[twoside,11pt]{article}
\usepackage{jart}
\usepackage[utf8]{inputenc}          % اول این باید بیاد
\usepackage{graphicx}
\usepackage{amsmath}
\usepackage{amssymb}
\usepackage{amsthm}
\usepackage{enumitem}
\usepackage{bbm}
\usepackage[l2tabu,orthodox]{nag}
\usepackage[all,error]{onlyamsmath}
\usepackage[strict]{csquotes}
\usepackage{url}
\usepackage{amssymb,amsthm,mathrsfs}
\newcommand{\V}{{\bf{V}}}

\DeclareUnicodeCharacter{200E}{}     % بعد این بیاد

\begin{document}
	\title{Polynomial extension of Van der Waerden's Theorem near zero}
	
	\Author{Gh.Ghadimi$^\dag$,  M.A.Tootkaboni$^{\ddag}$\correspond}
	
	\Address{$^{\dag}$Department of Pure Mathematics,	Faculty of Mathematical Sciences\\
		$^{\ddag}$Department of Pure Mathematics,	Faculty of Mathematical Sciences}
	
	\Email{ghadirghadimi@gmail.com, tootkaboni.akbari@gmail.com}
	\Markboth{Gh.Ghadimi,  M.A.Tootkaboni}{Van der Waerden's Theorem near zero}
	
	\Abstract{Let $S$ be a dense subring of the real numbers. In this paper we prove a polynomial version of Van der Waerden's theorem near zero. In fact, we prove that if $p_1,\ldots,p_m \in \mathbb{Z}[x]$ are polynomials such that $p_i(0) = 0$ and there exists $\delta > 0$ such that $p_i(x) > 0$ for every $x \in (0,\delta)$ and for every $i=1,\ldots , m$. Then for any finite partition $\mathcal{C}$ of \( S\cap(0,1) \) and every sequence $f:\mathbb{N}\to S\cap(0,1)$ satisfying $\sum_{n=1}^\infty f(n)<\infty$, there exist a cell $C \in \mathcal{C}$, an element $a \in S$, and $F \in P_f(\mathbb{N})$ such that 
		\[
		\{ a + p_i(\sum_{t \in F} f(t)) : i = 1,2,\ldots,m \} \subseteq C.
		\]   }
	
	\Keywords{Ultrafilter, Stone-\v Cech Compactification, Minimal ideal, Piecewise syndetic set, Partial semigroup.}
	\AMS{2010}{05D10, 22A15, 54D35.}
	
	 %%%%%%%%%%%%%%%%

\section{Introduction}
In 1996, an extension of van der Waerden's theorem to polynomials was formulated by Vitaly Bergelson and Alexander Leibman, see \cite{Bergelson1996}. As a consequence of their proof, we can say that for every finite subset $F$ of $\mathbb{Z}[x]$ without constant term and for every finite coloring $\mathcal{C}$ of $\mathbb{Z}$, the set $\{x + p(y) : p \in F\}$ is monochromatic. That is, there exist $a, b \in \mathbb{Z}$ and $C \in \mathcal{C}$ such that $\{a + P(b) : P \in F\} \subseteq C$. For more details, see \cite{Bergelson2004} and \cite{McCutcheon99}. In this paper, we prove polynomial van der Waerden's Theorem near zero. Theorem \ref{thm3.11} is our main results.

\section{Preliminary}

Let \( (S, \cdot) \) be a semigroup, and let $\beta S$ denote the collection of all ultrafilters on $S$. For any subset \( A \subseteq S \), define $\overline{A} = \{p \in \beta S \mid A \in p\}$. The collection \(\{\overline{A} \mid A \subseteq S\}\) forms a basis for a topology on \(\beta S\), with respect to which $\beta S$ is a compact Hausdorff space. This space is known as the Stone–Čech compactification of \(S\).

The operation "\(\cdot\)" on \(S\) can be uniquely extended to $\beta S$ so that \((\beta S, \cdot)\) becomes a compact right topological semigroup; that is, for any \(p \in \beta S\), the function \(r_p \colon \beta S \to \beta S\) defined by \(r_p(q) = q \cdot p\) is continuous. Moreover, \(S\) is contained in the topological center of \(\beta S\), meaning that for every \(x \in S\), the map \(\lambda_x \colon \beta S \to \beta S\) defined by \(\lambda_x(q) = x \cdot q\) is continuous.

For \(p, q \in \beta S\) and \(A \subseteq S\), we have \(A \in p \cdot q\) if and only if \(\{x \in S \mid x^{-1} \cdot A \in q\} \in p\), where \(x^{-1} \cdot A = \{y \in S \mid x \cdot y \in A\}\).

A nonempty subset \(I\) of a semigroup \((S, \cdot)\) is called a \textit{left ideal} if \(S \cdot I = \{s \cdot i : s \in S, i \in I\} \subseteq I\); a \textit{right ideal} if \(I \cdot S \subseteq I\); and a \textit{two-sided ideal} (or simply an \textit{ideal}) if it is both a left and right ideal. A \textit{minimal left ideal} is a left ideal that contains no proper left ideal. A \textit{minimal right ideal} is defined similarly. For more details, see \cite{Hindman}.

\subsection{Partial semigroups}
Now we introduce some fundamental concepts required for our work. First, we focus on the notion of a partial semigroup. For further details, see \cite{Hindman,McCutcheon}.

Let $S$ be a non-empty set, and let $*$ be a binary operation defined on a subset $D \subseteq S \times S$. The pair \((S, *)\) is called a \textit{partial semigroup} if, for all \(x, y, z \in S\), the associativity condition \((x * y) * z = x * (y * z)\) holds in the sense that if either side is defined, then so is the other, and they are equal.

We say that \(x * y\) is defined if \((x, y) \in D\). For each \(x \in S\), define:
\[
R_S(x) = \{s \in S : x * s \text{ is defined} \}, \quad L_S(x) = \{s \in S : s * x \text{ is defined} \}.
\]
When we write \(S * s = \{t * s : t \in S\}\) for some \(s \in S\), it should not be confusing; in fact, \(S * s = L_S(s) * s\).

A nonempty subset \(I \subseteq S\) is called a \textit{left ideal} of \(S\) if \(y * x \in I\) for all \(x \in I\) and \(y \in L_S(x)\). Similarly, \(I\) is a \textit{right ideal} if \(x * y \in I\) for all \(x \in I\) and \(y \in R_S(x)\). We say \(I\) is an \textit{ideal} if it is both a left and a right ideal.

A subset \(L \subseteq S\) is called a \textit{minimal left ideal} if \(L\) is a left ideal of \(S\) and, for every left ideal \(J \subseteq L\), we have \(J = L\). A \textit{minimal right ideal} is defined analogously.

An element \(p \in S\) is called \textit{idempotent} if \(p * p = p\). The set of all idempotents is denoted by \(E(S)\).

Now, we recall some basic properties of partial semigroups (see, e.g., \cite{Bergelson5}).

\begin{definition}
	Let \((S, *)\) be a partial semigroup.\\
	(a) For \(H \in P_f(S)\), define \(R_S(H) = \bigcap_{s \in H} R_S(s)\).\\
	(b) We say that \((S, *)\) is \textit{adequate} if \(R_S(H) \neq \emptyset\) for all \(H \in P_f(S)\).\\
	(c) Define \(\delta S = \bigcap_{H \in P_f(S)} \overline{R_S(H)}\).
\end{definition}

By Theorem 2.10 in \cite{McCutcheon}, \(\delta S \subseteq \beta S\) is a compact right topological semigroup.

If \((S, *)\) is a partial semigroup, then for every \(s \in S\) and \(A \subseteq S\), define
\[
s^{-1} A = \{t \in R_S(s) : s * t \in A\}.
\]

\begin{definition}
	Let \((S, *)\) be a partial semigroup.\\
	(a) For \(a \in S\) and \(q \in \overline{R_S(a)}\), define \(a * q = \{A \subseteq S : a^{-1} A \in q\}\).\\
	(b) For \(p, q \in \beta S\), define
	\[
	p * q = \{A \subseteq S : \{a \in S : a^{-1} A \in q\} \in p\}.
	\]
\end{definition}

By Lemma 2.7 in \cite{McCutcheon}, if \((S, *)\) is an adequate partial semigroup, then for every \(a \in S\) and \(q \in \overline{R_S(a)}\), we have \(a * q \in \beta S\). Moreover, if \(p \in \beta S\), \(q \in \delta S\), and \(a \in S\), then \(R_S(a) \in p * q\) whenever \(R_S(a) \in p\). In addition, for every \(p, q \in \delta S\), we have \(p * q \in \delta S\).

\begin{lemma} \label{1.10} (Lemma 1.10 in \cite{McLeod1})
	Let \(T\) be an adequate partial semigroup, and let \(S \subseteq T\) be an adequate partial semigroup under the inherited operation. Then the following statements are equivalent: \\
	(a) \(\delta S \subseteq \delta T\). \\
	(b) For all \(y \in T\), there exists \(H \in P_f(S)\) such that \(\bigcap_{x \in H} R_S(x) \subseteq R_T(y)\). \\
	(c) For all \(F \in P_f(T)\), there exists \(H \in P_f(S)\) such that \(\bigcap_{x \in H} R_S(x) \subseteq \bigcap_{x \in F} R_T(x)\).
\end{lemma}

\begin{definition} \label{1.11}
	Let \( T \) be a partial semigroup. Then \( S \) is an \textbf{adequate partial subsemigroup} of \( T \) if and only if \( S \subseteq T \),  
	\( S \) is an adequate partial semigroup under the inherited operation, and for all \( y \in T \), there exists \( H \in P_f(S) \) such that  
	$\bigcap_{x \in H} R_S(x) \subseteq R_T(y)$.	  
\end{definition}

\begin{theorem} \label{1.23}(Lemma 1.23 in \cite{McLeod1})
	Let $T$ be an adequate partial semigroup, let $S$ be an adequate partial subsemigroup of $T$ and assume that $S$ is an ideal of $T$. Then $\delta S$ is an ideal of $\delta T$. In particular, $K(\delta S) = K(\delta T)$.
\end{theorem}

\begin{definition}{\label{2.6}}
	Let $S\subseteq (0,1)$. For every $x,y\in S$, we define operation $"\dot{+}"$ on $S$ as follows:
	\[
	x\dot{+}y=x+y\,\,\,\,\,\,\,\text{ if and only if }\,\,\,\,\,\,\,x+y\in S.
	\]
	If $(S,\dot{+})$ is an adequate partial semigroup and $S$ is dense in $(0,1)$, then $(S,\dot{+})$ is called partially near zero semigroup. 
\end{definition}
\begin{lemma}
	Let $T$ be a dense subsemigroup of $((0,+\infty),+)$ and let $S=T\cap (0,+\infty)$. Then the following statements hold:\\
	(a) $(S,\dot{+})$ is a partial semigroup.\\
	(b) For every $x\in S$, $R_S(x)=(0,1-x)\cap S$.\\
	(c) $(S,\dot{+})$ is partially near zero semigroup.
\end{lemma}
\begin{proof}
	It is obvious.
\end{proof}
In \cite{ba}, \cite{Hin-Lead}, \cite{Ghosh}, and \cite{Pa}, the relation between piecewise syndetic sets and \( J \)-sets has been studied separately for partial semigroups and for ultrafilters near zero. Here, we present an alternative proof based on the algebraic properties of adequate partial semigroups.
\begin{definition}
	Let $S\subseteq (0,1)$. A sequence $f:\mathbb{N}\to S$ is called a partially sequence if and only if for each $H\in P_f(\mathbb{N})$, $\sum_{t\in H}f(t) \in (0,1)$. 
\end{definition}

It is obvious that, if $\sum_{n\in\mathbb{N}}f(n)<1$, then $f:{\mathbb{N}}\to S$ is a partially sequence. For every partially sequence $f$ in $S$ and for every $F\in P_f(\mathbb{N})$, we define $T_F^f:(0,1)\to (0,1)$ by $T_F^f(s) = s \dotplus \sum_{t\in F}f(t)$ for every $s\in R_S(\sum_{t\in F}f(t))$. We say that $A\subseteq (0,1)$ is a partially $J$-set if for every finite subset $\{f_1,\dots,f_k\}$ of partially sequences, there exist $F\in P_f(\mathbb{N})$ and $a\in \bigcap_{i=1}^k R_S(\sum_{t\in F}f_i(t))$ such that 
\[
(T_F^{f_1}(a),\dots,T_F^{f_k}(a)) \in \times_{i=1}^k A. 
\] 
Since $((0,1),\dotplus)$ is a commutative adequate partial semigroup, $(\times_{i=1}^k(0,1),\dotplus)$ is a commutative adequate partial semigroup for every $k\in\mathbb{N}$, where $\dotplus$ on $\times_{i=1}^k(0,1)$ is defined by 
\[
(x_1,\dots,x_k) \dotplus (y_1,\dots,y_k) = (x_1 \dotplus y_1, \dots, x_k \dotplus y_k),
\]
for every $(x_1,\dots,x_k),(y_1,\dots,y_k)\in\times_{i=1}^k(0,1)$. 

\begin{definition}
	Fix $k\in\mathbb{N}$ and let $f_1,f_2,\dots,f_k$ be partial sequences in $(0,1)$.\\
	(a) Define $\Delta^k_J = \{(s,s,\dots,s) \in \times_{i=1}^k J : s \in J\}$ for every non-empty set $J$.\\	
	(b) Define 
	\[
	E = \left\{ (T_F^{f_1}(a),\dots,T_F^{f_k}(a)) : a \in \bigcap_{i=1}^k R_S\left(\sum_{t\in F}f_i(t)\right), F \in P_f(\mathbb{N}) \cup \{\emptyset\} \right\}
	\]
	and 
	\[
	I = \left\{ (T_F^{f_1}(a),\dots,T_F^{f_k}(a)) : F \in P_f(\mathbb{N}), a \in \bigcap_{i=1}^k R_S\left(\sum_{t\in F}f_i(t)\right) \right\},
	\]	 
	where $T_\emptyset^f(a) = a$ for every $a\in S$ and for every $f:{\mathbb{N}}\to S$.
\end{definition}

\begin{theorem} \label{recurrent}
	Let $f_1,f_2,\dots,f_k$ be partially sequences. Then the following statements hold:\\
	(a) $(E,\dotplus)$ is an adequate partial subsemigroup of $\times_{i=1}^k(0,1)$, where 
	\[
	(T_F^{f_1}(a),\dots,T_F^{f_k}(a)) \dotplus (T_G^{f_1}(b),\dots,T_G^{f_k}(b)) = (T_{F\cup G}^{f_1}(a \dotplus b),\dots,T_{F\cup G}^{f_k}(a \dotplus b)),
	\]  
	when $F\cap G=\emptyset$ and $T_F^{f_i}(b) \in R_S(T_F^{f_i}(a))$ for every $i=1,2,\dots, k$. Moreover, if $Y=\times_{i=1}^k\delta (0,1)$, then $\delta E = \bigcap_{x\in E} cl_Y R_E(x)$ is a compact subsemigroup of $Y$.\\
	(b) $(I,\dotplus)$ is an adequate partial subsemigroup of $E$. Also, $\delta I$ is an ideal of $\delta E$ and $K(\delta I) = K(\delta E)$.\\
	(c) $K(\delta E) = K\left( \times_{i=1}^k\delta (0,1) \right) \cap \delta E$.	
\end{theorem}

\begin{proof}
	\begin{enumerate}
		\item[(a)] Naturally for every $(x_1,\dots,x_k)\in\times_{i=1}^k(0,1)$, we have 
		\begin{align*}
		R_{E}((x_1,\dots,x_k)) 
		&= \left\{ y \in \times_{i=1}^k(0,1) : (x_1,\dots,x_k) \dotplus y \text{ is well defined} \right\} \\
		&= \times_{i=1}^k (0, 1 - x_i) \\
		&\supseteq \times_{i=1}^k (0, 1 - x), 
		\end{align*}
		where $x = \min\{x_1,\dots,x_k\}$. So, $(E,\dotplus)$ is an adequate partial semigroup. Now, pick $x=(x_1,\dots,x_k)\in\times_{i=1}^k(0,1)$, and let\linebreak $\delta = \max\{x_1,x_2,\dots,x_k\}$. For $(\delta,\dots,\delta)\in E$ we have $\times_{i=1}^k(0,1-\delta) \subseteq R_{E}((x_1,\dots,x_k))$, and hence $E$ is an adequate partial subsemigroup of $\times_{i=1}^k (0,1)$ by Definition \ref{1.11}. Therefore $\delta E = \bigcap_{x\in E} \overline{R_{E}(x)}$ is a compact right topological semigroup and $\delta E$ is a compact subsemigroup of $\times_{i=1}^k\delta(0,1)$.
		
		\item[(b)] Similar to part (a), it can be readily observed that $I$ is an adequate partial subsemigroup of $E$ and also $I$ is an ideal of $E$. By Theorem \ref{1.23}, $\delta I$ is an ideal of $\delta E$ and $K(\delta I) = K(\delta E)$.
		
		\item[(c)] Let $Y = \times_{i=1}^k\delta (0,1)$, then $K(Y) = \times_{i=1}^k K(\delta (0,1))$ by Theorem 2.23 in \cite{Hindman}. Pick $\overline{p} = (p,\dots,p) \in K(Y)$, and let $U$ be a neighborhood of $\overline{p}$. For every $x \in E$ we have $U \cap R_{E}(x) \neq \emptyset$, and hence $\overline{p} \in \overline{R_{E}(x)} \neq \emptyset$. So $\overline{p} \in K(Y) \cap \delta E$. Since $K(Y) \cap \delta E \neq \emptyset$, $K(\delta E) = K(Y) \cap \delta E$ by Theorem 1.65 in \cite{Hindman}, which completes the proof. 
	\end{enumerate}
\end{proof}

Now, we establish Theorem 14.8.3 from \cite{Hindman}, utilizing the results previously derived for partial semigroups. The proof follows the same structure as the proof in \cite{Hindman}, with the only difference being that we utilize the concept of partial semigroups. A subset \( A \) of an adequate partial semigroup \( S \) is said to be \textit{piecewise syndetic} if and only if \( \overline{A} \cap K(\delta S) \neq \emptyset \); see Definition 3.3 in \cite{McLeod1}.

\begin{theorem} \label{Jset}
	Let \( (S,\dot{+}) \) be a partially near zero semigroup, and let \( A \) be a partially piecewise syndetic subset of \( S \). Then, \( A \) is a partially \( J \)-set.
\end{theorem}

\begin{proof}
	Since $A$ is a piecewise syndetic subset of $S$, pick $p \in \overline{A} \cap K(\delta S)$. Let $\{f_1,\dots,f_k\}$ be a partially sequence and define  
	\[
	E = \left\{ (T_F^{f_1}(a),\dots,T_F^{f_k}(a)) : a \in S, F \in P_f(\mathbb{N}) \cup \{\emptyset\} \right\},
	\]
	where $T_\emptyset^f(a) = a$ for every $a \in S$ and $f:{\mathbb{N}}\to S$. Then $E$ is an adequate partially semigroup and so $\overline{p} = (p,\dots,p) \in K\left( \times_{i=1}^k\delta S \right) \cap \delta E = K(\delta E)$ by Theorem \ref{recurrent}. So, by Theorem 2.8 in \cite{Hindman}, there exists an idempotent $\eta \in K(\delta E)$ such that $\eta \dotplus \overline{p} = \overline{p}$. Since $\times_{i=1}^k A \in \overline{p}$, we have 
	\[
	\mathcal{G} = \left\{ x \in E : -x \dotplus \left( \times_{i=1}^k A \right) \in \overline{p} \right\} \in \eta,
	\]
	and hence 
	\[
	\bigcap_{i=1}^k \left( -\left(s \dotplus \sum_{t\in F} f_i(t)\right) \dotplus A \right) \cap A \in p
	\]
	for some $s \in S$ and $F \in P_f(\mathbb{N})$. So there exist $a \in A$, $s \in S$ and $F \in P_f(\mathbb{N})$ such that $a \dotplus s \dotplus \sum_{t\in F} f_i(t) \in A$ for every $i=1,\dots,k$. In fact,
	\[
	(T_F^{f_1}(a \dotplus s),\dots,T_F^{f_k}(a \dotplus s)) \in \times_{i=1}^k A.
	\] 
	%This completes the proof. 
\end{proof}

\section{\bf Symbolic one Variable Polynomials Space}
In this section, we introduce the space of symbolic polynomials. We will show that every polynomial with real coefficients is the image of a symbolic polynomial.

For \( k \in \mathbb{N} \), we fix the symbols \( 1_0, 1_1, \ldots, 1_k \), and construct strings of the form \( (1_0)( 1_{1})( 1_{2}) \ldots (1_{i}) \), where $0<i\leq k$. 
\begin{definition}
	
	(a) \( \Gamma_k = \{ (a_0 1_0)(a_1 1_1) \cdots (a_i 1_i) : i \in [1, k] \text{ and for } t \in [0,i],\ a_t \in \mathbb{R} \} \), where $[j,i]=\{j,j+1,\ldots,i\}$ for every $j<i$ and $i,j\in\mathbb{Z}$.
	
	(b) If \( x = (a_0 1_0)(a_1 1_1) \cdots (a_i 1_i) \) and \( y = (b_0 1_0)(b_1 1_1) \cdots (b_j 1_j) \), are members of \( \Gamma_k \), then 
	$
	x = y $ if and only if $i = j$ and for each  $t \in [0, i]$, $a_t = b_t$.
	
	(c) For \( x = (a_0 1_0)(a_1 1_1) \cdots (a_i 1_i) \in \Gamma_k \), define \( \iota(x) = a_0 \) and \( l(x) = i \). $\iota(x)$ is the first natural integer that appears in $x$ and $l(x)$ is length of $x$.
	
	(d) For \( x, y \in \Gamma_k \), $ x $ and $y$ are irreducible if and only if  $l(x)\neq l(y)$  or \( \iota(x)\neq\iota(y) \). Otherwise, $x$ and $y$ are called compatible. 
	
	(e) Let $x_1,\ldots,x_m\in \Gamma_k$, we say that $\{x_1,\ldots,x_m\}$ is irreducible set if $x_i$ and $x_j$ are irreducible for every distinct $i,j\in[1,m]$. 
	
	(h) For \( x, y \in \Gamma_k \), \( x \prec y \) if and only if \( l(x) < l(y) \) and if \( l(x)=l(y) \), then \( \iota(x) < \iota(y) \). We say that $x\preceq y$ if and only if $x\prec y$ or $x=y$.
\end{definition} 
\begin{lemma}{\label{3.222}}
	$(\Gamma_k,\preceq)$ is totally ordered set.
\end{lemma}
\begin{proof}
	The proof is routine.	
\end{proof}
\begin{lemma}{\label{{3.33}}}
	Let $\{x_1,\ldots,x_k\}$ be a finite subset of $\Gamma_k$. Then the following statement hold: \\
	(a) If $\{x_1,\ldots,x_k\}$ is irreducible set, then there exists a unique permutation $\sigma:[1,n]\to[1,n]$ such that 
	\[
	x_{\sigma(1)}\prec x_{\sigma(2)}\prec\cdots\prec x_{\sigma(n)}.
	\]
	(b) Let $	x_{1}\prec x_{2}\prec\cdots\prec x_{n}$. Then $\{x_1,\ldots,x_n\}$ is irreducible set.
\end{lemma}
\begin{proof}
	By Lemma \ref{3.222} and Definition $\prec$, it is obvious.
\end{proof}
\begin{definition}{\label{3.22}}
	For every $x = (a_0 1_0)(a_1 1_1) \cdots (a_i 1_i)$ and $ y = (b_0 1_0)\cdots (b_j 1_j)$ in $\Gamma_k$, we define $x+y$ as follow:\\
	(a) If $x$ and $y$ are compatible, i.e., $\iota(x)=\iota(y)$ and $l(x)=l(y)$, define $x+y=(a_01_0)((a_1+b_1) 1_1) \cdots ((a_i+b_i) 1_i)$, and  \\
	(b) if $x$ and $y$ are irreducible we just write $x+y$. 
\end{definition}
In fact, the $"+"$ concatenates the two strings $x$ and $y$, and if the two strings are compatible, it assigns a simple string to them. We are now ready to define $"+"$ for a finite number of $\Gamma_k$ elements.
\begin{definition}{\label{3.55}}
	(a)	For every $n\in\mathbb{N}$, define 
	\[
	\Gamma_k^n=\left\{ x_1 + x_2 + \cdots + x_n : x_{1}\prec x_{2}\prec\cdots\prec x_{n}, \{x_1,\ldots,x_n\}\subseteq\Gamma_k\right\}.
	\]
	(b) Define $V_k =\cup_{i=1}^\infty \Gamma_k^n$, where $\Gamma_k^1=\Gamma_k$. $V_k$ is called \textbf{one variable symbolic polynomials space}. If \( \gamma = x_1 + x_2 + \cdots + x_n \) as in the definition of \( V_k \), then \( x_1, x_2, \ldots, x_n \) are the \textit{terms of \( \gamma \)}. The set of terms of $\gamma$ is denoted by $Term(\gamma)$.\\
	(c) Let \( \gamma = x_1 + x_2 + \cdots + x_n \) and \( \mu = y_1 + y_2 + \cdots + y_m \) be two elements in \( \Gamma^n_k \). We say that \( \gamma = \mu \) if and only if $Term(\gamma)=Term(\mu)$. 
\end{definition}
\begin{definition}{\label{3.366}}
	Given \( \gamma = x_1 + x_2 + \cdots + x_n \) and \( \mu = y_1 + y_2 + \cdots + y_m \) in \( V_k \) written as in the definition of \( V_k \), \( \gamma + \mu \) is defined as follows.
	\begin{enumerate}
		
		\item[(a)] 	Given \( t \in [1,n] \), there is at most one \( s \in [1,m] \) such that \( \iota(x_t) = \iota(y_s) \) and \( l(x_t) = l(y_s) \). 
		
		\item[(b)] For every \( t \in [1,n] \), if there is no such \( s \in \{1, 2, \ldots, m\} \) such that (a) holds. In fact, if  $\{x_1,\ldots,x_n\}\cup\{y_1,\ldots,y_m\}$ is a irreducible subset of $\Gamma_k$.
	\end{enumerate}
	(a) If there is such \( s \), assume that $x_t = (a_0 1_0)(a_1 1_1) \cdots (a_i 1_i)$ and $ y_s = (b_0 1_0)(b_1 1_1) \cdots (b_i 1_i)$ such that \( b_0 = a_0 \). Let $z_t = (a_0 1_0)((a_1 + b_1) 1_1) \cdots ((a_i + b_i) 1_i)$. Having chosen \( z_1, z_2, \ldots, z_d \) for $0\leq d\leq min\{m,n\}$. By definition of $V_k$, $\{z_1,\ldots,z_d\}$ is irreducible subset of $\Gamma_k$. Now, let
	\[
	B = \left\{ y_s :\{x_1\ldots,x_n\}\cup\{y_s\}\text{ is irreducible subset of }\Gamma_k. \right\}.
	\]
	(Possibly \( B = \emptyset \).) 	Let \( q = |B| \) and let \( w_1, w_2, \ldots, w_{d+q} \) enumerate \linebreak \( \{z_1, z_2, \ldots, z_d\} \cup B \) by $\prec$. Then we define
	\[
	\gamma + \mu = \mu + \gamma = w_1 + w_2 + \cdots + w_{d+q}.
	\]
	It is obvious that $\{w_1,w_2,\ldots,w_{d+q}\}$ is irreducible set.
	
	(b) If $\{x_1,\ldots,x_n\}\cup\{y_1,\ldots,y_m\}$ is a irreducible subset of $\Gamma_k$, let $w_1,\ldots,w_{n+m}$ enumerate $\{x_1,\ldots,x_n\}\cup\{y_1,\ldots,y_m\}$ by $\prec$. Then we define 
	\[
	\gamma+\mu=\mu+\gamma=w_1+w_2+\cdots+w_{n+m}.
	\]
\end{definition}

\begin{remark}{\label{3.44}}
	Let $x=x_1+\cdots+x_n$ and $y=y_1+\cdots+y_m$ be two elements of $V_k$. Then there exist finite subsets $I(x,y)$ of $Term(x,y)=Term(x)\cup Term(y)$ and $C(x,y)$ of $Term(x)\times Term(y)$ such that\\
	(a) $C(x,y)=\{(a,b)\in Term(x)\times Term(y):l(a)=l(b),\iota(a)=\iota(b)\}$ and \\
	(b) $I(x,y)=\{u\in Term(x,y): \{u\}\cup\{a+b:(a,b)\in C(x,y)\}\mbox{ is irreducible}\}$ is an irreducible subset of $Term(x)\cup Term(y)$. \\
	Then $E(x,y)=I(x,y)\cup \{a+b:(a,b)\in C(x,y)\}$ is irreducible, so let $z_1,\ldots,z_d$ enumerate $E(x,y)$ by $\prec$. Now define 
	\[
	x+y=z_1+\ldots+z_d.
	\]  
	Obviously, $I(x,y)=I(y,x)$ and $C(x,y)=C(y,x)$, and hence $E(x,y)=E(y,x)$. So $x+y=y+x$ for every $x,y\in V_k$.
\end{remark}
\begin{theorem}{\label{3.66}}
	Let $k\in\mathbb{N}$. Then $(V_k,+)$ is a commutative semigroup.
\end{theorem}
\begin{proof}
	By Definition \ref{3.366}, $V_k$ is closed under operation $+$, and by Remark \ref{3.44}, $(V_k,+)$ is commutative.
	
	Now we prove that $(x+y)+z=x+(y+z)$ for every $x,y,z\in V_k$, by induction on $|Term(z)|$.\\
	Let $z\in \Gamma_k$, and let $x,y\in V_k$ be two arbitrary elements. For $x,y\in V_k$, 
	\[
	Term(x+y)=A_x\cup A_y\cup\{a+b:(a,b)\in C(x,y)\}.
	\]
	Notice that $A_x=I(x,y)\cap Term(x)$, $A_y=I(x,y)\cap Term(y)$ and $\{a+b:(a,b)\in C(x,y)\}$ are disjoint. Let $Term(x+y)=\{u_1\prec u_2\prec\cdots\prec u_l\}$.\\
	{\bf{ Case 1:}}\\
	If there exists $u_i\in A_x$ such that $z$ and $u_i$ are compatible, then $(A_x\setminus\{u_i\})\cup\{u_i+z\}=A_{x,z}$ is irreducible, and also  $A_{x,z}\cup A_y\cup\{a+b:(a,b)\in C(x,y)\}$  and $Term(y)\cup\{z\}$  are irreducible. Therefore, we have 
	\[
	(x+y)+z=u_1+u_2+\cdots+(u_i+z)+\cdots+u_l=x+(y+z).
	\]
	{\bf{ Case 2:}}\\	
	If there exists $u_i\in A_y$ such that $z$ and $u_i$ are compatible, then $(A_y\setminus\{u_i\})\cup\{u_i+z\}=A_{y,z}$ is irreducible, and also  $A_{x}\cup A_{y,z}\cup\{a+b:(a,b)\in C(x,y)\}$  and $Term(y)\cup\{z\}$  are irreducible. Therefore we have 
	\[
	(x+y)+z=u_1+u_2+\cdots+(u_i+z)+\cdots+u_l=x+(y+z).
	\]
	{\bf{ Case 3:}}\\	
	If there exists $u_i\in \{a+b:(a,b)\in C(x,y)\}$ such that $z$ and $u_i$ are compatible, then there exist $a\in Term(x)$ and $b\in Term(y)$ such that $u_i=a+b$, $(a,z)\in C(x,z)$ and $(b,z)\in C(y,z)$.  Therefore we have 
	\[
	u_1\prec\cdots\prec u_{i-1}\prec u_i+z=a+b+z\prec u_{i+1}\prec \cdots\prec u_l.
	\]
	This implies that $(Term(y)\setminus{b})\cup\{b+z\}$ is irreducible. Therefore we have 
	\begin{align*}
	(x+y)+z&=u_1+u_2+\cdots u_{i-1}+((a+b)+z)+u_{i}\cdots+u_l\\
	&=u_1+u_2+\cdots u_{i-1}+(a+(b+z))+u_{i}\cdots+u_l\\
	&=x+(y+z).
	\end{align*}	
	{\bf{ Case 4:}}\\	
	If $z$ is not compatible with any member of $Term(x+y)$, then $Term(x+y)\cup\{z\}$ is irreducible. Therefore we will have
	\[
	u_1\prec u_2\prec\prec u_{i-1}\prec z\prec u_i\prec\cdots\prec u_l.
	\]
	Since $Term(y)\cup\{z\}$ is irreducible, implies that $(x+y)+z=x+(y+z)$.
	
	Now, assume that $|Term(z)|=n > 1$ and the statement is true for smaller sets, (induction hypothesis). Now let $z=z_1+\cdots+z_n$, let $x,y\in V_k$ be two arbitrary elements of $V_k$. Then, by the induction hypothesis,  we have	
	\begin{align*}
	(x+y)+z=&(x+y)+(z_1+z_2+\cdots+z_n)\\
	=&(x+y)+\left(z_1+(z_2+\cdots+z_n)\right)\\
	=&\left((x+y)+z_1\right)+(z_2+\cdots+z_n)\\
	=&\left(x+(y+z_1)\right)+(z_2+\cdots+z_n)\\
	=&x+\left((y+z_1)+(z_2+\cdots+z_n)\right)\\
	=&x+\left(y+(z_1+(z_2+\cdots+z_n))\right)\\
	=&x+\left(y+(z_1+z_2+\cdots+z_n)\right)\\
	&=x+(y+z).
	\end{align*}	
	Therefore, $"+"$ is an associative operation on $V_k$. 
	
\end{proof}

\begin{definition}
	For an element $\eta\in V_k$, we define 
	\[
	Ir(\{\eta\})=\{y\in V_k:Term(y)\cup Term(\eta)\mbox{ is an irreducible subset of }\Gamma_k\}.
	\] 
\end{definition}
\begin{lemma}{\label{366}}
	(a) Let $\eta\in V_k$, $Ir(\{\eta\})$ is a subsemigroup of $V_k$.\\ 
	(b) $\{Ir(\{\eta\}):\eta\in V_k\}$ has the finite intersection property.\\
	(c) For a finite subset $\{\eta_1,\ldots,\eta_m\}$ of $V_k$, $Ir(\{\eta_1,\ldots,\eta_m\})=\bigcap_{i=1}^mIr(\{\eta_i\})$ is a subsemigroup of $V_k$.
\end{lemma}
\begin{proof}
	(a) Let $L=\{\iota(v):v\in Term(\eta)\}$ then $\{x\in\Gamma_k:\iota(x)\notin L\}$ is a subset of $Ir(\{\eta\})$. Therefore $Ir(\{\eta\})$ is non-empty set. 
	
	Now, let $x,y\in Ir(\{\eta\})$. Then $Term(x+y)=\{u_1\prec u_2\prec\cdots\prec u_l\}$. Now, let $A=I(x,y)\cup \{a+b:(a,b)\in C(x,y)\}$ and let $A\cup Term(\eta)$ be not irreducible. So there exist $u\in A$ and $v\in Term(\eta)$ such that $u$ and $v$ are compatible. If $u\in I(x,y)$, we have a contradiction. So let $u=a+b$ for some $(a,b)\in C(x,y)$. Since $\iota(u)=\iota(a)=\iota(b)$ and $l(u)=l(a)=l(b)$, so $v$ and $a\in Term(x)$ are compatible. Therefore, we have a contradiction. Therefore $A\cup Term(\eta)$ is irreducible, and so $x+y\in Ir(\{\eta\})$. This completes our proof.\\
	\item[(b)] For every finite set $F \subseteq V_k$, let
	$A = \left\{ \iota(x) \mid x \in \cup_{\eta \in F} \mathrm{Term}(\eta) \right\}.$ Now define 
	\[
	B=\{x\in\Gamma_k:\iota(x)\notin A\}.
	\] 
	It is obvious that $B\subseteq \bigcap_{\eta\in F}Ir(\{\eta\})$, and so $\{Ir(\{\eta\}):\eta\in V_k\}$ has the finite intersection property. \\
	(c) It is obvious.
\end{proof}
\begin{definition}{\label{3.6}}
	For $ r = ( r_1, \ldots, r_k)\in \mathbb{R}^{k} $, $(a_0,a_1,\ldots,a_k) \in \mathbb{R}^{k+1}$ and $a=(a_01)(a_11_1)(a_21_2)\cdots(a_i1_i)$, we define 
	\[
	r\bullet a=(a_01_0)(r_1a_11_1)(r_2a_21_2)\cdots(r_ia_i1_i). 
	\]
	Also, for $x,y\in \Gamma$, we define $\quad r \bullet (x+y) =r\bullet x+r\bullet y$.
\end{definition}
\begin{lemma}
	Let $k\in \mathbb{N}$, then the following statements hold:\\
	(a) For every $r\in \mathbb{R}^k$ and $\eta\in \Gamma$, $r\bullet \eta$ is well defined. Also for every $\eta_1,\eta_2\in V_k$ and $r\in \mathbb{R}^k$, we have $r\bullet(\eta_1+\eta_2)=r\bullet\eta_1+r\bullet\eta_2$.\\
	(b) For every  $a,b\in \mathbb{R}^k$ and every $\eta\in V_k$, we have $(a+b)\bullet \eta=a\bullet\eta+b\bullet\eta$.
\end{lemma}
\begin{proof}
	The proof is routine.
\end{proof}

\begin{remark}
	Let $p(x)=\sum_{i=1}^ka_ix^i$ be a polynomial in $\mathbb{R}[x]$.  Define $P_x:V_k\to \mathbb{R}[x]$ by $P_x(a1_0)=a$ and $P_x(a1_i)=ax$ for every $i=1,2,\ldots,k$ and for every $a\in \mathbb{R}$. It is obvious that for every $u,v\in V_k$, if $u\in Ir(\{v\})$, then we have $P_x(u+v)=P_x(u)+P_x(v)$. Therefore
	\[
	P_x(\sum_{i=1}^k(a_i1_0)(1_1)(1_2)\cdots(1_i))=\sum_{i=1}^ka_ix^i.
	\]
	Therefore for every $p(x)\in\mathbb{R}[x]$, there exists $\eta\in \V_k$ such that $P_x(\eta)=p(x)$.
\end{remark}

\section{Symbolic Polynomials Space Near Zero}
\begin{definition}
	Let $k\in\mathbb{N}$ and let $\mathbb{L}$ be a subring of $\mathbb{R}$. We define function $\pi:V_k\to \mathbb{R}$ with the following properties:
	\begin{enumerate}[label=(\alph*)]
		\item 
		\(\pi(1_i) = 1 \quad \forall i=0,\ldots,k\).
		\item 
		\(\pi(a1_i) = a \quad \forall a\in \mathbb{L}, \quad \forall i=0,1,\ldots,k\).
		\item 
		\(\pi\left(\prod_{j=0}^i (a_j 1_j)\right) = \prod_{j=0}^i \pi(a_j 1_j) = \prod_{j=0}^i a_j \quad \forall \prod_{j=0}^i (a_j 1_j)\in V_k\).
		\item 
		$\pi(a+b) = \pi(a) + \pi(b)$ if $a\in \mathrm{Ir}(\{b\})$.
	\end{enumerate}
	Therefore if $a_0=b_0$, we will have 
	\[
	\pi\left(\prod_{j=0}^i (a_j 1_j) + \prod_{j=0}^i (b_j 1_j)\right) = \pi\left((a_01_0)\prod_{j=1}^i ((a_j+b_j) 1_j)\right) = a_0 \prod_{j=1}^i (a_j+b_j).
	\]
\end{definition}

\begin{definition}
	Let $k\in\mathbb{N}$ and let $(S,\dot{+})$ be partially near zero semigroup. The collection of all symbolic polynomials near zero is denoted by $V_k(0,S)$, and is defined by 
	\[
	V_k(0,S)=\{\eta\in V_k:\pi(\eta)\in S\}.
	\]
	We define $\pi_S(x)=\pi(x)$ for every $x\in V_k(0,S)$.
\end{definition}

\begin{theorem}
	Let $(S,\dot{+})$ be a partially near zero semigroup. For every $k\in\mathbb{N}$, $(V_k(0,S),\dotplus)$ is a commutative adequate partial semigroup, where $\eta_1 \dotplus \eta_2 = \eta_1 + \eta_2$ if $\pi(\eta_1 + \eta_2)\in S$.
\end{theorem}
\begin{proof}
	It is obvious that $(V_k(0,S),\dotplus)$ is a commutative adequate partial semigroup because for every $\eta\in V_k(0,S)$, we have 
	\begin{align*}
	R(\eta) &= \{\zeta\in V_k(0,S) : \eta \dotplus \zeta \text{ is well defined} \} \\
	&= \{\zeta\in V_k(0,S) : \pi(\eta + \zeta)\in S \} \\
	&\supseteq \{\zeta\in \mathrm{Ir}(\{\eta\}) : \pi(\eta) + \pi(\zeta)\in S \} \\
	&= \{\zeta\in \mathrm{Ir}(\{\eta\}) : \pi(\zeta)\in (0,1-\pi(\eta)) \cap S \} \\
	&= \pi_S^{-1}((0,1-\pi(\eta))) \cap \mathrm{Ir}(\{\eta\}).
	\end{align*}
	It is obvious that $\pi^{-1}((0,1-\pi(\eta)) \cap S) \cap \mathrm{Ir}(\{\eta\}) \neq \emptyset$ and for every $\eta_1,\eta_2\in V_k(0,S)$, we have 
	\[
	R({\eta_1}) \cap R({\eta_2}) \supseteq \pi_S^{-1}((0,1-\min\{\pi(\eta_1),\pi(\eta_2)\})) \cap \mathrm{Ir}(\{\eta_1,\eta_2\}) \neq \emptyset.
	\]
\end{proof}

Now, consider the product space $\times_{i=1}^m V_k(0,S)$ equipped with the pointwise operation $\dotplus$. It is clear that $\left(\times_{i=1}^m V_k(0,S), \dotplus\right)$ is a commutative adequate partial semigroup.

\begin{definition}
	Let $(S,\dot{+})$ be a partially near zero semigroup. Let $\{\eta_1,\ldots,\eta_m\}$ be an arbitrary non-empty finite subset of $V_k$.\\
	(a) Define $\mathrm{Ir} = \mathrm{Ir}(\{\eta_i\}_{i=1}^m) = \bigcap_{i=1}^m \mathrm{Ir}(\{\eta_i\}) \cap V_k(0,S)$.\\
	(b) Define 
	\[
	V = V(\{\eta_i\}_{i=1}^m) = \left\{ \left(x \dotplus r \bullet \eta_1, \ldots, x \dotplus r \bullet \eta_m \right) : x \in \mathrm{Ir}, r \in \Delta_S^k \right\} \cup \Delta_{\mathrm{Ir}}^m.
	\]
	(c) Define 
	\[
	I = I(\{\eta_i\}_{i=1}^m) = \left\{ \left(x \dotplus r \bullet \eta_1, \ldots, x \dotplus r \bullet \eta_m \right) : x \in \mathrm{Ir}, r \in \Delta_S^k \right\}.
	\]
	(d) For every $i=1,\ldots,m$, let
	\[
	T_i = \{x \dotplus r \bullet \eta_i : r \in \Delta_S^k, x \in \mathrm{Ir}\} \cup \mathrm{Ir}, 
	\]
	and define $T = T(\{\eta_i\}_{i=1}^m) = \bigcup_{i=1}^m T_i$.
\end{definition}

Let $Y = \times_{i=1}^m \delta V_k(0,S)$. Since for every $(x_1,\ldots,x_m) \in V$, we have 
\[
R_V((x_1,\ldots,x_m)) = \times_{i=1}^m R_T(x_i),
\] 
we define:
\[
\delta V = \bigcap_{x \in V} \operatorname{cl}_Y R_V(x)
\quad \text{and} \quad
\delta I = \bigcap_{x \in I} \operatorname{cl}_Y R_I(x).
\]

\begin{lemma} \label{asli}
	Let $(S,\dot{+})$ be a partially near zero semigroup. Let $\{\eta_1,\ldots,\eta_m\}$ be an arbitrary non-empty finite subset of $V_k$. The following statements hold:\\
	(a) Let $T = T(\{\eta_i\}_{i=1}^m)$. Then $(T,\dotplus)$ is a commutative adequate  partial subsemigroup of $(V_k(0,S),\dotplus)$.\\
	(b) Let $V = V(\{\eta_i\}_{i=1}^m)$. Then $(V,\dotplus)$ is a commutative adequate  partial subsemigroup of $\times_{i=1}^m T$.\\
	(c) Let $I = (I(\{\eta_i\}_{i=1}^m),\dotplus)$. Then $(I,\dotplus)$ is an adequate partial subsemigroup of $(V,\dotplus)$. Also, $I$ is an ideal of $V$ and $K(\delta V) = K(\delta I)$.\\
	(d) $\mathrm{Ir}$ is an adequate partial subsemigroup of $T$, $\delta (\mathrm{Ir}) = \delta T$.\\
	(e) Let $Y = \times_{i=1}^m \delta S_i$. Then $K(\delta V) = K(Y) \cap \delta V$ and $K(\delta V)\subseteq \delta T$.
\end{lemma}

\begin{proof}
	(a) It is obvious that $T$ is a non-empty subset of $V_k(0,S)$, and $(T,\dotplus)$ is a commutative partial semigroup. For every $x\in T$, 
	\[
	\pi_S^{-1}\left((0,1-\pi(x))\right) \cap \mathrm{Ir}(\{\eta_i\}_{i=1}^m) \cap T \subseteq R_T(x)
	\]
	if $x = a + r \bullet \eta_i$ for some $a \in \mathrm{Ir}$ and some $i \in \{1,\ldots,m\}$, and if $x \in \mathrm{Ir}$, then $ R_T(x) = \pi_S^{-1}\left((0,1-\pi(x))\right) \cap T$. Therefore $T$ is an adequate partial semigroup.
	
	Now, pick $y \in V_k(0,S)$. Since $\pi_S(y) < 1$, there exists $x \in \pi_S^{-1}((\pi_S(y),1)) \cap T$. Therefore $R_T(x) = \pi_S^{-1}((0,1-\pi(x))) \cap T \subseteq \pi_S^{-1}((0,1-\pi(y)))$. So, by Definition \ref{1.11}, $(T,\dotplus)$ is an adequate partial subsemigroup of $V_k(0,S)$.
	
	(b) It is obvious that $(V,\dotplus)$ is a commutative partial semigroup and by part (a), since 
	\[
	\bigcap_{i=1}^l R_V((u^i_1,\ldots,u^i_m)) = \times_{i=1}^m R_T(\{u^1_i,u_i^2,\ldots,u_i^l\})
	\]
	the proof is complete. Now similar to the proof of part (a), it follows that $V$ is a partial subsemigroup of $\times_{i=1}^m V_k(0,S)$.
	
	(c) The proof follows similarly to part (b) and by Theorem \ref{1.23}.
	
	(d) It is obvious that $(\mathrm{Ir},\dotplus)$ is an adequate partial semigroup. Now, pick $y \in V_k(0,S)$. So there exists $x \in \mathrm{Ir}$ such that $\pi_S(y) < \pi_S(x) < 1$. Therefore 
	\[
	R_{\mathrm{Ir}}(y) = \pi_S^{-1}((0,1 - \pi_S(y))) \cap \mathrm{Ir} \supseteq \pi_S^{-1}((0,1 - \pi_S(x))) \cap \mathrm{Ir} \neq \emptyset.
	\]
	This implies that $\mathrm{Ir}$ is an adequate partial subsemigroup of $V_k(0,S)$, and so by Lemma \ref{1.10}, $\delta \mathrm{Ir} \subseteq \delta V_k(0,S)$. Now, by Theorem \ref{1.23}, it follows that $\delta \mathrm{Ir} = \delta T$.
	
	(e) Let $Y = \times_{i=1}^m \delta S_i$. Let \( p \in K(T)\), so \(\overline{p} = (p, p, \ldots, p) \in K(Y) \). We claim that $\overline{p} \in \delta V$. Let $U$ be a neighborhood of $\overline{p}$ and let $x \in V$, so there exist $C_1, \ldots, C_m \in p$ such that $\times_{t=1}^m C_i \subseteq U$. Now pick $a \in \bigcap_{i=1}^m C_i\cap \mathrm{Ir}$ and so $\overline{a} = (a,\ldots,a) \in U \cap R_V(x)$. This implies that $ \overline{p} \in K(Y) \cap \delta V$ and so by Theorem 1.65 in \cite{Hindman}, it follows that $K(\delta V) = K(Y) \cap \delta V$, and so $K(\delta V) \subseteq \delta I$. 
\end{proof}

\begin{theorem}
	Let $(S,\dot{+})$ be a partially near zero semigroup. Let $\{\eta_1,\ldots,\eta_m\}$ be an arbitrary non-empty finite subset of $V_k$ and let $\mathrm{Ir} = \bigcap_{i=1}^m \mathrm{Ir}(\{\eta_i\})$. Let $p \in K(\delta \mathrm{Ir})$ and $A \in p$, then there exist \( a \in \mathrm{Ir} \) and \( r \in \Delta_S^k \) such that  
	\[
	\{ a + r \bullet \eta_1, \ldots, a + r \bullet \eta_m \} \subseteq A.
	\] 
\end{theorem}
\begin{proof}
	For $\{\eta_1,\ldots,\eta_m\}$, define $V = V(\{\eta_i\}_{i=1}^m)$ and $I = I(\{\eta_i\}_{i=1}^m)$. By Lemma \ref{asli}, $I$ is an adequate partial subsemigroup of $V$ and also $I$ is an ideal of $V$, so by Lemma \ref{1.10} and Theorem \ref{1.23}, $K(\delta I) = K(\delta V)$.
	
	Now if $A \in p$, then \( I \cap \times_{t=1}^{m} A \neq \emptyset \). Pick \( z \in I \cap \times_{t=1}^{m} A \), and choose \( a \in \mathrm{Ir} \) and \( r \in \Delta_S^k \) such that  
	\[
	z = \left( a + r \bullet \eta_1, \dots, a + r \bullet \eta_m \right) \in \times_{t=1}^{m} A.
	\] 
	%This completes the proof.
\end{proof}

Let $(S,\dot{+})$ be partially near zero semigroup and let $f:\mathbb{N}\to S$ be a sequence such that $\sum_{n\in\mathbb{N}}f(n)<1$. Pick $\eta = \sum_{i=1}^k (c_i1)(1_1)(1_2)\cdots(1_i) \in V_k$ and $F \in P_f(\mathbb{N})$. Then we define $T_F^{\eta}f:V_k(0,S)\to V_k(0,S)$ by
\begin{align*}
T_F^{\eta}f(x) &= x \dotplus \sum_{t\in F} \left( r(t) \bullet \left( \sum_{i=1}^k (c_i1)(1_1)(1_2)\cdots(1_i) \right) \right) \\
&= x \dotplus \sum_{t\in F} \left( \sum_{i=1}^k (c_i1)(f(t)1_1)(f(t)1_2)\cdots(f(t)1_i) \right) \\
&= x \dotplus \sum_{i=1}^k \left( (c_i1) \left( \left( \sum_{t\in F} f(t) \right) 1_1 \right) \cdots \left( \left( \sum_{t\in F} f(t) \right) 1_i \right) \right),
\end{align*}
where $x\in V_k(0,S)$. The domain of the map \( T_F^\eta f\) is
\[
\mathrm{Dom}(T_F^\eta f) = \pi_S^{-1}\left( \left(0, 1 - \pi ( T_F^\eta f(x) \right)\right) \cap V_k(0,S).
\]
It can be shown that $T_F^{\eta}f \circ T_G^{\eta}f(x) = T_{F \cup G}^{\eta}f(x)$ if $F \cap G = \emptyset$ for every $F,G \in P_f(\mathbb{N})$.
Now, let $r(t) = (f(t),\ldots,f(t))$ be a $k$-tuple vector, let $\{\eta_1,\ldots,\eta_m\}$ be a finite subset of $ V_k(0,S)$ and let $\mathrm{Ir} =\cap_{i=1}^m \mathrm{Ir} ({\eta}_i)$.
We define 
\[
T_i = \{ T_F^{\eta_i}f(x) : F \in P_f(\mathbb{N}), x \in \mathrm{Ir} \} \cup \mathrm{Ir}
\]
for every $i=1,\ldots,m$ and let $T_f = \bigcup_{i=1}^m T_i$. For every $i \in \{1,\ldots,m\}$ and for every $x,y\in T_i$, we define $x\dot{+}y=x+y$ if $y \in \mathrm{Ir}$, and if $x=T_F^{\eta_i}f(x_1)\in T_i$, $y=T_G^{\eta_i}f(x_2) \in T_i$ for $F,G\in P_f(\mathbb{N})$ and $x_1,x_2\in Ir$, then $x \dotplus y = x + y$ if $F \cap G = \emptyset$.

We extend operation $\ddot{+}$ on $T_f$. For every $x,y \in T_f$, we define
\[
x \ddot{+} y = x \dotplus y \quad \text{if and only if} \quad \exists i \in \{1,\ldots,m\} \quad x,y \in T_i.
\]
It is obvious that $(T_f,\ddot{+})$ is a commutative adequate partial semigroup.

Now assume that   
\[
V_f = \left\{ \left( T_F^{\eta_1}f(x), \ldots, T_F^{\eta_m}f(x) \right) : x \in T_f, F \in P_f(\mathbb{N}) \right\} \cup \Delta_{\mathrm{Ir}}^m.
\]
Also, define 
\[
I_f = \left\{ \left( T_F^{\eta_1}f(x), \ldots, T_F^{\eta_m}f(x) \right) : x \in T_f, F \in P_f(\mathbb{N}) \right\}.
\]

Now, consider the operator $\ddot{+}$ as component-wise addition on $V_f$, i.e., for two elements $u = (u_1, u_2, \dots, u_m)$ and $v = (v_1, v_2, \dots, v_m)$ in $V_f$, we define
\[
u \ddot{+} v = (u_1 \ddot{+} v_1, u_2 \ddot{+} v_2, \dots, u_m \ddot{+} v_m).
\]

\begin{lemma} \label{asli_2}
	Let $(S,\dot{+})$ be a partially near zero semigroup. Let $\{\eta_1,\ldots,\eta_m\}$ be an arbitrary non-empty finite subset of $V_k$ and let $f:\mathbb{N}\to S$ be a partially sequence. The following statements hold:\\
	(a) $(T_f,\ddot{+})$ is a commutative adequate partial subsemigroup of $(V_k(0,S),\dotplus)$.\\
	(b) $(V_f,\ddot{+})$ is a commutative adequate partial subsemigroup of $\times_{i=1}^m V_k(0,S)$.\\
	(c) $(I_f,\ddot{+})$ is an adequate partial subsemigroup of $(V_f,\ddot{+})$. Also, $I_f$ is an ideal of $V_f$ and $K(\delta V_f) = K(\delta I_f)$.\\
	(d) $\mathrm{Ir}$ is an adequate partial subsemigroup of $T_f$. and so $\delta \mathrm{Ir} = \delta T_f$.\\
	(e) Let $Y = \times_{i=1}^m \delta T_f$. Then $K(\delta V_f) = K(Y) \cap \delta V_f$.
\end{lemma}
\begin{proof}
	The proof is similar to Lemma \ref{asli}.
\end{proof}

\begin{theorem} \label{5.9}
	Let $(S,\dot{+})$ be a partially near zero semigroup. Pick $k>1$. Let \( \eta = (\eta_1, \eta_2, \ldots, \eta_m) \) be an $m$-tuple in $\times_{i=1}^m V_k$ without constant term and let $f:\mathbb{N}\to S$ be a partially sequence, Then for every finite partition $\mathcal{C}$ of $V_k(0,S)$, there exist $x \in \mathrm{Ir}(\{\eta_1,\ldots,\eta_m\})$, $F \in P_f(\mathbb{N})$ and $C \in \mathcal{C}$ such that:
	\[
	\{ T_F^{\eta_1}f(x), T_F^{\eta_2}f(x), \ldots, T_F^{\eta_m}f(x) \} \subseteq C.
	\]
\end{theorem}

\begin{proof}
	Define $V_f$ and $I_f$. By Lemma \ref{asli_2}, $K(\delta V_f) = K(Y) \cap \delta V_f$ and $K(\delta I_f) = K(\delta V_f)$. If $\mathcal{C}$ is a finite partition of $V_k(0,S)$, there exists $p \in K(\delta V_f)$ and $C \in \mathcal{C}$ such that $C \in p$, so $\times_{i=1}^m C \cap R_{I_f}(x) \neq \emptyset$. Therefore, there exist $a \in \mathrm{Ir}$ and $F \in P_f(\mathbb{N})$ such that $\left\{ T_F^{\eta_1}f(a), \ldots, T_F^{\eta_m}f(a) \right\} \subseteq C$. 
\end{proof}

The following theorem is a version of Van der Waerden Polynomial version near zero.

\begin{theorem}\label{thm3.11}
	Let $(S,\dot{+})$ be a partially near zero semigroup. Let $p_1,\ldots,p_m \in \mathbb{Z}[x]$ be polynomials such that $p_i(0) = 0$ and there exists $\delta > 0$ such that $p_i(x) > 0$ for every $x \in (0,\delta)$ for every $i=1,\ldots ,m$. Then for any finite partition $\mathcal{C}$ of \( S \) and every partially sequence $f$, there exist a cell $C \in \mathcal{C}$, $a \in S$, and $F \in P_f(\mathbb{N})$ such that 
	\[
	\{ a + p_i(\sum_{t \in F} f(t)) : i = 1,2,\ldots,m \} \subseteq C.
	\] 
\end{theorem}

\begin{proof}
	Let $p_j(x) = \sum_{i=1}^{k_j} a_{i j} x^{i}$ for $j=1,\ldots,m$ and let $k = \max\{k_1,\ldots,k_m\}$. Pick $f:\mathbb{N}\to S$ such that $\sum_{n \in \mathbb{N}} f(n) < 1$ and define 
	\[
	\eta_j = \sum_{i=1}^{k_j} (a_{i j} 1_0) (1_1) \cdots (1_i) \in V_k.
	\] 
	Now assume that $\mathcal{C}$ is a finite partition for $S$. Then $\{\pi^{-1}(C) : C \in \mathcal{C}\}$ is a finite partition for $V_k(0,S)$. By Theorem \ref{5.9}, there exist $x \in \mathrm{Ir}(\{\eta_1,\ldots,\eta_m\})$ and $F \in P_f(\mathbb{N})$ such that for some $C \in \mathcal{C}$, implies that
	\[
	\{ T_F^{\eta_1}f(x), T_F^{\eta_2}f(x), \ldots, T_F^{\eta_m}f(x) \} \subseteq \pi^{-1}(C).
	\]
	This implies that $\{\pi(T_F^{\eta_1}f(x)), \pi(T_F^{\eta_2}f(x)), \ldots, \pi(T_F^{\eta_m}f(x))\} \subseteq C$. Since 
	\begin{align*}	
	\pi(T_F^{\eta_j}f(x)) &= \pi\left( x + \sum_{i=1}^k \left( (a_{i j} 1_0) \left( \left( \sum_{t \in F} f(t) \right) 1_1 \right) \cdots \left( \left( \sum_{t \in F} f(t) \right) 1_i \right) \right) \right) \\
	&= \pi(x) + \sum_{i=1}^{k_j} a_{i j} \left( \sum_{t \in F} f(t) \right)^i \\
	&= \pi(x) + p_j\left( \sum_{t \in F} f(t) \right).
	\end{align*}
	Let $a = \pi(x)$, so we have:
	\[
	\{ a + p_i(\sum_{t \in F} f(t)) : i = 1,2,\ldots,m \} \subseteq C.
	\]
\end{proof}
You can find some results of the Theorem  \ref{thm3.11} in a special case in \cite{De}. For example, see Theorems 2.3 and 2.4 follow easily from the Theorem  \ref{thm3.11}.

% Please complete affiliations of all co-authors (using not abbreviated names and surnames).

\noindent Ghadir Ghadimi\\  %Name Surname
ghadirghadimi@gmail.com\\
%ORCID: \url{https://orcid.org/0000-0000-0000-0000} \bigskip % ORCID identifier (optional)
\noindent {\small
	\noindent University of Guilan\\
	Department of Pure Mathematics,	Faculty of Mathematical Sciences\\
	Rasht, Guilan, Iran.
}\bigskip

\noindent M. A. Tootkaboni (corresponding author)\\
tootkaboni.akbari@gmail.com\\
%ORCID: \url{https://orcid.org/0000-0000-0000-0000} \bigskip
\noindent {\small
	\noindent University of Guilan\\
	Department of Pure Mathematics,	Faculty of Mathematical Sciences\\
	Rasht, Guilan, Iran.
}\bigskip

\end{document}